\documentclass[10pt,a4paper,reqno]{amsart}
\usepackage{amsfonts,amsthm,latexsym,amsmath,amssymb,amscd,epsfig}
\usepackage{graphics,graphicx}
\ifx\undefined\xspace \usepackage{xspace} \fi

\newcommand{\nc}{\newcommand}
\renewcommand{\frak}{\mathfrak}
\providecommand{\cal}{\mathcal}
\renewcommand{\bold}{\mathbf}
\numberwithin{equation}{section}

\newcommand{\pfname}{Proof.}

                      {$\square$\vskip\medskipamount\par}

                      {$\square$\vskip\medskipamount\par}

\newtheorem{thm}{Theorem}[section]

\newtheorem{thmnn}{Theorem}

\newtheorem{cor}[thm]{Corollary}

\newtheorem{proposition}[thm]{Proposition}

\newtheorem{lemma}[thm]{Lemma} 

\theoremstyle{definition}
\newtheorem{definition}[thm]{Definition}

\theoremstyle{definition}

\newtheorem{remark}[thm]{Remark}

\nc{\Theorem}[1]{Theorem~{#1}}
\nc{\Th}[1]{({\sl Th.}~#1)}
\nc{\Thd}[2]{({\sl Th.}~{#1} {#2})}
\nc{\Theorems}[2]{Theorems~{#1} and ~{#2}}
\nc{\Thms}[2]{({\it Thms. ~{#1} and ~{#2}})}
\nc{\Lemmas}[2]{Lemma~{#1} and ~{#2}}

\nc{\manga}[6]{({\it Thms. ~ #1, ~ #2, ~ #3,\\ ~ #4, ~ #5, ~ #6})}

\nc{\Prop}[1]{({\sl Prop.}~{#1})}
\nc{\Proposition}[1]{Proposition~{#1}}
\nc{\Propositions}[2]{Propositions~{#1} and ~{#2}}
\nc{\Props}[2]{({\sl Props.}~{#1} and ~{#1})}
\nc{\Cor}[1]{({\sl Cor.}~{#1})}
\nc{\Corollary}[1]{Corollary~{#1}}
\nc{\Corollaries}[2]{Corollaries~{#1} and ~{#2}}
\nc{\Definition}[1]{Definition~{#1}}
\nc{\Defn}[1]{({\sl Def.}~{#1})}
\nc{\Lemma}[1]{Lemma~{#1}} 
\nc{\Lem}[1]{({\sl Lem.} ~{#1})} 
\nc{\Eq}[1]{equation~({#1})}
\nc{\Equation}[1]{Equation~({#1})}
\nc{\Section}[1]{Section~{#1}}
\nc{\Sections}[1]{Sections~{#1}}
\nc{\Sec}[1]{({\sl Sec.} ~{#1})} 
\nc{\Chapter}[1]{Chapter~{#1}}
\nc{\Chapt}[1]{({\sl Ch.}~{#1})}

\nc{\Ex}[1]{{\sl Ex.}~{#1}}
\nc{\Exa}[1]{{\sl Example}~{#1}}
\nc{\Example}[1]{{\sl Example}~{#1}}
\nc{\Examples}[1]{{\sl Examples}~{#1}}
\nc{\Exercise}[1]{{\sl Exercise}~{#1}}

\nc{\Rem}[1]{({\sl Rem.}~{#1})}
\nc{\Remark}[1]{{\sl Remark}~{#1}}
\nc{\Remarks}[1]{{\sl Remarks}~{#1}}
\nc{\Note}[1]{{\sl Note}~{#1}}

\nc{\Conjecture}[1]{Conjecture~{#1}}
\nc{\Claim}[1]{Claim~{#1}}

\nc \Proof{{  \it Proof. }}

\nc{\xmu}{\mu}
\nc{\w}{\omega}

\nc \Ab{{\ensuremath{\bold A}}}
\nc \ab{{\ensuremath{\bold a}}}
\nc \bb{{\ensuremath{\bold b}}}
\nc \cb{{\ensuremath{\bold c}}}
\nc \Bb{{\ensuremath{\bold B}}}
\nc \Gb{{\ensuremath{\bold G}}}
\nc \Qb{{\ensuremath{\bold Q}}}
\nc \Rb{{\ensuremath{\bold R}}} \nc \Cb{{\ensuremath{\bold C}}} 
\nc \Eb{{\ensuremath{\bold E}}}
\nc \eb{{\ensuremath{\bold e}}}
\nc \Db{{\ensuremath{\bold D}}}
\nc \Fb{{\ensuremath{\bold F}}}
\nc \ib{{\ensuremath{\bold i}}}
\nc \jb{{\ensuremath{\bold j}}}
\nc \kb{{\ensuremath{\bold k}}}
\nc \nb{{\ensuremath{\bold n}}}
\nc \rb{{\ensuremath{\bold r}}}
\nc \Pb{{\ensuremath{\bold P}}}
\nc \pb{{\ensuremath{\bold p}}}
\nc \SPb{{\ensuremath{\bold {SP}}}}
\nc \Zb{{\ensuremath{\bold Z}}} 
\nc \zb{{\ensuremath{\bold z}}} 
\nc \gb{{\ensuremath{\bold g}}} 
\nc \fb{{\ensuremath{\bold f}}} 
\nc \ub{{\ensuremath{\bold u}}} 
\nc \vb{{\ensuremath{\bold v}}} 
\nc \yb{{\ensuremath{\bold y}}} 
\nc \xb{{\ensuremath{\bold x}}} 
\nc \xib{{\ensuremath{\bold \xi}}} 
\nc \Nb{{\ensuremath{\bold N}}} 
\nc \Hb{{\ensuremath{\bold H}}} 
\nc \wb{{\ensuremath{\bold w}}} 
\nc \Wb{{\ensuremath{\bold W}}} 
\nc \syz{{\mathbf {syz}}}
\nc \bnoll{{\ensuremath{\bold 0}}} 

\nc \mf{\frak m} \nc \mh{\hat{\m}} 
\nc \nf{\frak n}
\nc \Of{\frak O}
\nc \rf{\frak r}
\nc \mufr{{\mathbf \mu}}
\nc \hf{\frak h} 
\nc \qf{\frak q} 
\nc \bfr{\frak b} 
\nc \kfr{\frak k} 
\nc \pfr{\frak p} 
\nc \af{\frak a }
\nc \cf{\frak c }
\nc \sfr{\frak s} 
\nc \ufr{\frak u} 
\nc \g{\frak g} 
\nc \gA{\g_{\Ao}} 
\nc \lfr{\frak l}
\nc \afr{\frak a}
\nc \gfh{\hat {\frak g}}
\nc \gl{\frak { gl }}
\nc \Sl{\frak {sl}}
\nc \SU{\frak {SU}}
\nc{\Homf}{\frak{Hom}}

\newcommand{\on}{\operatorname}
\nc\hankel{\on {Hankel}}
\nc\row{\on {row\ }}
\nc\nullity{\on {nullity }}
\nc\col{\on {col\ }}
\nc\rowm{\on {Row \ }}
\nc\loc{\on {lc \ }}
\nc\nullo{\on {null\ }}
\nc\Nul{\on {Nul\ }}
\nc \Ann {\on {Ann }}
\nc \Ass {\on {Ass \ }}
\nc \Coker {\on {Coker}}
\nc \Co{\on C}
\nc \Homo{\on {Hom}}
\nc \Ker {\on {Ker}}
\nc \omod{\on {mod}}
\nc \No {\on N}
\nc \NN {\on {NN}}
\nc \NGo {\on {NG}}
\nc \Oo {\on O}
\nc \ch {\on {ch}}
\nc \rko {\on {rk}}
\nc \Sing {\on {Sing\ }}
\nc \Reg {\on {Reg}}
\nc \CoI {\on {CI}}
\nc \CoM {\on {CM}}
\nc \Gor {\on {Gor}}
\nc \Type {\on {Type}}
\nc \can {\on {can}}
\nc \Top {\on {T}}
\nc \Tr {\on {Tr}}
\nc \rel {\on {rel}}
\nc \tr {\on {tr}}
\nc \sgn {\on {sgn }}
\nc \trdeg {\on {tr.deg}}
\nc \codim {\on {codim }}
\nc \coht {\on {coht}}
\nc \divo {\on {div \ }}
\nc \coh {\on {coh}}
\nc \Clo {\on {Cl}}
\nc \embdim{\on {embdim}}
\nc \embcodim{\on {embcodim \ }}
\nc \qcoh {\on {qcoh}}
\nc \grad {\on {grad}\ }
\nc \grade {\on {grade}}
\nc \hto {\on {ht}}
\nc \depth {\on {depth}}
\nc \prof {\on {prof}}
\nc \reso{\on {res}}
\nc \ind{\on {ind}}
\nc \prodo{\on {prod}}
\nc \coind{\on {coind}}
\nc \Con{\on {Con}}
\nc \Crit{\on {Crit}}
\nc \Der{\on {Der}}
\nc \Char{\on {Char}}
\nc \Ch{\on {Ch}}

\nc \Ext{\on {Ext}}
\nc \Eo{\on {E}}
\nc \End{\on {End}}
\nc \ad{\on {ad}}
\nc \Ad{\on {Ad}}
\nc \gr{\on {gr}}
\nc \Fo{\on {F}}
\nc \Gr{\on {Gr}}
\nc \Go{\on {G}}
\nc \GFo{\on {GF}}
\nc \Glo{\on {Gl}}
\nc \Ho{\on {H}}
\nc \CMo{\on {\CM}}
\nc \SCM{\on {SCM}}
\nc \hol{\on {hol}}
\nc{\sgd}{\on{sgd}}
\nc \supp{\on {supp}}
\nc \ssupp{\on {s-supp}}
\nc \singsupp{\on {singsupp}}
\nc \msupp{\on {msupp}}
\nc \spec{\on {spec}}
\nc \spano{\on {span }}
\nc \Span{\on {Span }}
\nc \Max{\on {Max}}
\nc \Min{\on {Min}}
\nc \Mod{\on {Mod}}
\nc \Rad {\on {Rad}}
\nc \rad {\on {rad}}
\nc \rank {\on {rank}}
\nc \range {\on {range}}
\nc \Slo{\on {SL}}
\nc \soc {\on {soc}}
\nc \Irr {\on {Irr}}
\nc \Imo {\on {Im}}
\nc \SSo{\on {SS}}
\nc \lub{\on {lub}}
\nc \gldim{\on {gl.d.}}
\nc \pdo{\on {p.d.}} 
\nc \ido{\on {i.d.}} 
\nc \dSSo{\dot {\SSo}}
\nc \So{\on S}
\nc \Io{\on I}
\nc \Jo{\on J}
\nc \jo{\on j}
\nc \Ko{\on K}
\nc \PBW{\Ac_{PBW}}
\nc \Ro{\on R}
\nc \To{\on T}
\nc \Ao{\on A}

\nc \Do{{\on D}}
\nc \Bo{\on B}
\nc \Po{\on P}
\nc \Qo{\on Q}
\nc \Zo{\on Z}
\nc \U{\on U}
\nc \wt{\on {wt}}
\nc \Uh{\hat {\U}}
\nc \T{\on T}
\nc \Lo{\on L}
\nc{\dop}{\on d}
\nc{\eo}{\on e}
\nc{\ado}{\on{ad}}
\nc{\Tot}{\on{Tot}}
\nc{\Aut}{\on{Aut}}
\nc{\sinc}{\on {sinc}}

\nc{\overrightleftarrows}[2]{\overset{#1}{\underset{#2}{\rightleftarrows}}}

\nc{\CCF}{\cal{CF}}
\nc{\CDF}{\cal{DF}}
\nc{\CHC}{\check{\cal C}}

\nc{\Cone}{\on{Cone}}
\nc{\dec}{\on{dec}}
\nc{\Diff}{\on{Diff}}
\nc{\dirlim}{\underset{\to}{\on{lim}}}
\nc{\dpar}{\partial}
\nc{\GL}{\on{GL}}
\nc{\CGr}{\cal{G}r}
\nc{\pr}{\on{pr}}
\nc{\semid}{|\!\!\!\times}
\nc{\Hom}{\on{Hom}}
\nc \RHom{\on {RHom}}

\nc \Proj{\mathrm {Proj\ }}
\nc \proj{\mathrm {proj}}
\nc{\Id}{\on{Id}}
\nc{\id}{\on{id}}
\nc{\Ima}{\on{Im}}
\nc{\invtimes}{\underset{\gets}{\otimes}}
\nc{\invlim}{\underset{\gets}{\on{lim}}}
\nc{\Lie}{\on{Lie}}
\nc{\re}{\on{Re }}
\nc{\Pic}{\on{Pic }}
\nc{\LPic}{\on{LPic }}
\nc{\Sch}{\on{Sch}}
\nc{\Sh}{\on{Sh}}
\nc{\Set}{\on{Set}}
\nc{\spo}{\on{sp\  }}
\nc{\Spec}{\on{Spec}}
\nc{\mSpec}{\on{mSpec}}
\nc{\Specb}{\bold {Spec}}
\nc{\Projb}{\bold {Proj}}
\nc{\Specan}{\on{Specan}}
\nc{\Spo}{\on{Sp}}
\nc{\Spf}{\on{Spf}}
\nc{\sym}{\on{sym}}
\nc{\symm}{\on{symm}}
\nc{\rop}{\on{r}}
\nc{\Td}{\on{Td}}
\nc{\Tor}{\on{Tor}}

\nc{\Artin}{\cal{A}rtin}
\nc{\Dgcoalg}{\cal{D}gcoalg}
\nc{\Dglie}{\cal{D}glie}
\nc{\Ens}{\cal{E}ns}
\nc{\Fsch}{\cal{F}sch}
\nc{\Groupoids}{\cal{G}roupoids}
\nc{\Holie}{\cal{H}olie}
\nc{\Mor}{\cal{M}or}

\nc{\CF}{\ensuremath{\cal{F}}}
\nc \Kc{\ensuremath{\cal K}}
\nc \Lc{{\ensuremath{\cal L}}}
\nc \lcc{{\mathcal l}} 
\nc \CC{{\ensuremath{\cal C}}} 
\nc \Cc{{\ensuremath {\cal C}}}
\nc \Pc{{\ensuremath{\cal P}}}
\nc \Dc{\ensuremath{\mathcal D}}
\nc \Ac{{\ensuremath{\cal A}}} 
\nc \Bc{{\ensuremath{\cal B}}}
\nc \Ec{{\ensuremath{\cal E}}}
\nc \Fc{{\ensuremath{\cal F}}}
\nc \Mcc{{\ensuremath{\cal M}}} 
\nc \hM{\hat{\Mcc}} 
\nc \bM{\bar {\Mcc}} 
\nc\hbM{\hat{\bar \Mcc}}  
\nc \Nc{{\ensuremath{\cal N}}}
\nc \Hc{{\ensuremath{\cal H}}} 
\nc \Ic{{\ensuremath{\cal I}}} 
\nc \Oc{\ensuremath{{\cal O}}}
\nc \Och{\hat{\cal O}} 
\nc \Sc{{\ensuremath{{\cal S}}}}
\nc \Tc{\ensuremath{{\cal T}}} 
\nc \Vc{{\ensuremath{{\cal V}}}} 
\nc{\CA}{{\ensuremath{{\cal A}}}}
\nc{\CB}{{\ensuremath{{\cal B}}}}
\nc{\C}{{\ensuremath{{\cal F}}}}
\nc{\Gc}{{\ensuremath{{\cal G}}}}
\nc{\CH}{\ensuremath{\mathcal H}}
\nc{\CI}{{\ensuremath{{\cal I}}}}
\nc{\CM}{{\ensuremath{{\cal M}}}}
\nc{\CN}{{\ensuremath{{\cal N}}}}
\nc{\CO}{{\ensuremath{{\cal O}}}}
\nc{\Rc}{{\ensuremath{{\cal R}}}}
\nc{\CT}{{\ensuremath{\mathcal T}}}
\nc{\CU}{\ensuremath{{\cal U}}}
\nc{\CV}{\ensuremath{{\cal V}}}
\nc{\CZ}{\ensuremath{{\cal Z}}}
\nc{\Homc}{\ensuremath{{\cal {Hom}}}}

\nc{\fa}{\frak{a}}
\nc{\fA}{\frak{A}}
\nc{\fg}{\frak{g}}
\nc{\fh}{\frak{h}}
\nc{\fI}{\frak{I}}
\nc{\fK}{\frak{K}}
\nc{\fm}{\frak{m}}
\nc{\fP}{\frak{P}}
\nc{\fS}{\frak{S}}
\nc{\ft}{\frak{t}}
\nc{\fX}{\frak{X}}
\nc{\fY}{\frak{Y}}

\nc{\bF}{\bar{F}}
\nc{\bCP}{\bar{\cal{P}}}
\nc{\bm}{\mbox{\bf{m}}}
\nc{\bT}{\mbox{\bf{T}}}
\nc{\hB}{\hat{B}}
\nc{\hC}{\hat{C}}
\nc{\hP}{\hat{P}}
\nc{\htest}{\hat P}

\nc{\nen}{\newenvironment}
\nc{\ol}{\overline}
\nc{\ul}{\underline}
\nc{\ra}{\to}
\nc{\lla}{\longleftarrow}
\nc{\lra}{\longrightarrow}
\nc{\Lra}{\Longrightarrow}
\nc{\Lla}{\Longleftarrow}
\nc{\Llra}{\Longleftrightarrow}
\nc{\hra}{\hookrightarrow}
\nc{\iso}{\overset{\sim}{\lra}}

\nc{\dsize}{\displaystyle}
\nc{\sst}{\scriptstyle}
\nc{\tsize}{\textstyle}
\nen{exa}[1]{\label{#1}{\bf Example.\ } }{}

\nen{rem}[1]{\label{#1}{\em Remark.\ } }{}
\nen{exer}[1]{\label{#1}{\em Exercise.\ } }{}

\newcommand {\bC} {\mathbb C}

\newcommand {\bZ} {\mathbb Z}
\newcommand{\bN}{\mathbb N}
\newcommand {\D} {\mathcal D}

\begin{document}
\numberwithin{equation}{section}
\title[Decomposition of plane D-modules]{decomposition of D-modules over a  hyperplane arrangement in the plane}
\author{Tilahun Abebaw and Rikard B\o gvad.}
\address{Department of Mathematics, Addis Ababa University and Stockholm University}\email{tabebaw@math.aau.edu.et, abebaw@math.su.se}
\address{Department of Mathematics, Stockholm University, Sweden} \email{rikard@math.su.se}
\date{}
\maketitle

\begin{abstract} Let $\alpha_{1},\alpha_{2}...,\alpha_{m}$ be linear forms
defined on $\bC^{n}$ and
$\mathcal{X}=\bC^{n}\setminus\cup_{i=1}^{m}V(\alpha_{i}),$ where
$V(\alpha_{i})=\{p\in\bC^{n}:\alpha_{i}(p)=0\}$. The coordinate ring
$O_{\mathcal{X}}$ of $\mathcal{X}$ is a holonomic $A_{n}$-module,
where $A_{n}$ is the n-th Weyl algebra and since holonomic
$A_{n}$-modules have finite length, $O_{\mathcal{X}}$ has finite
length. We consider a ''twisted'' variant of this $A_{n}$-module
which is also holonomic. Define $\rm{M_{\alpha}^{\beta}}$ to be the
free rank 1 $\bC[x]_{\alpha}$-module on the generator
$\alpha^{\beta}$ (thought as a multivalued function ), where
$\alpha^{\beta}=\alpha_{1}^{\beta_{1}}...\alpha_{m}^{\beta_{m}}$ and
the multi-index $\beta=(\beta_{1},...,\beta_{m})\in\bC^{m}$. 
Our main result is the computation of the number of decomposition
factors of $\rm{M_{\alpha}^{\beta}}$ and their description when
$n=2$.
\end{abstract}
\section{Introduction}\label{s-1}

\subsection{Definition of the module $\rm{M_{\alpha}^{\beta}}$}\label{ss-11}
Let $\alpha_{i}:\bC^{n}\longrightarrow\bC,$ $ i=1,2,...,m$
be linear forms and  $H_{i}=\{P\in\mathbb{C}^{n}:\alpha_{i}(P)=0\}$ the corresponding hyperplanes. If we let
$X=\mathbb{C}^{n}\setminus \cup_{i=1}^{m}H_{i}$ be the complement of the central hyperplane arrangement that $H_i, i=1,...m$ define, then the coordinate
ring of $X$ is the localization
$\mathbb{C}[x_{1},\dots,x_{n}]_{\alpha}$, where
$\alpha=\prod\limits_{i=1}^{m}\alpha_{i}$. This is a module over the Weyl Algebra $A_{n}$. Consider now for varying values of the
complex parameters $\beta=(\beta_{1},...,\beta_{m})\in \bC^{n}$, the multivalued
function
$$\alpha^{\beta}=\alpha_{1}^{\beta_{1}}...\alpha_{m}^{\beta_{m}}.$$
(we will throughout this
paper use the above multi-index notation. Also we will use $\bC[x]=\bC[x_{1},...,x_{n}]$.) The  $A_{n}$-module generated by $\alpha^{\beta}$ as a $\bC[x]_\alpha$-module can be abstractedly described in the following obvious way as a twisted version of $\bC[x]_\alpha$.
\begin{definition}\label{Def11}The module $\rm{M_{\alpha}^{\beta}}$, where $\beta=(\beta_{1},...,\beta_{m})\in \bC^{m}$ is (as a $\bC[x]$-module) the free rank 1 $\bC[x]_{\alpha}$-module on the generator $\alpha^{\beta}$. It is furthermore an $A_{n}$-module if we define
 $$\partial_{j}\alpha^{\beta}=\sum_{i=1}^{i=m}\frac {\beta_i\partial_{j}(\alpha_i)}{\alpha_i}\alpha^{\beta}$$  for $j=1,2,...,n$ and extend this to an action of the whole $A_{n}$ on $\rm{M_{\alpha}^{\beta}}$ .
\end{definition}

We may think of this as the fibres of a flat family of $A_{n}[\beta_1,...,\beta_n]$-modules, where now $\beta_i$ is a free scalar variable (i.e. commuting with the derivations).
The problem which we consider in this paper, and solve in the plane
case is to find the decomposition factors $\rm{DF(M_{\alpha}^{\beta})}$ of $\rm{M_{\alpha}^{\beta}}$, and the number $c(\rm{M_{\alpha}^{\beta}})$ of them.  A starting point for us was the intriguing fact --- see \cite{CDCP1}---that $c(\rm{M_{\alpha}^{\beta}})$ is easily expressible in terms of combinatorial data of the hyperplane arrangement when $\beta=0$.  

\subsection{Motivating example}\label{ss12}The case where $m=n=1$, and the $A_{1}$-module $\rm{M_{\alpha}^{\beta}}=\bC[x]_{x}x^{\beta}$ is easy to analyze.

\begin{proposition}\label{Prop11}(i) If $\beta\in\bZ$, then $c(\rm{M_{\alpha}^{\beta})}=2$.\\
(ii) If $\beta\in\bC\setminus\bZ$, then $M_{\alpha}^{\beta}$ is a simple $A_{1}$-module, so $c(M_{\alpha}^{\beta})=1$.
\end{proposition}
\begin{proof} By definition $\rm{M_{\alpha}^{\beta}}=\mathbb{C}[x]_{x}x^{\beta}\cong\oplus_{i\in \mathbb{Z}}\mathbb{C}x^{\beta+i}$.\\
$(i)$ If $\beta\in \mathbb{Z}$, then clearly
$\rm{M_{\alpha}^{\beta}}\cong \mathbb{C}[x]_{x}$.  It is an easy exercise to see that
$$\bC[x]\hookrightarrow\bC[x]_{x}\twoheadrightarrow\bC[x]_{x}/\bC[x]$$
 and that $\bC[x]$, $\bC[x]_{x}/\bC[x]$ are simple $\bC\langle x,\partial_{x}\rangle$-module. \\(ii) Suppose  $\beta \in \mathbb{C}\setminus\mathbb{Z}$. We have first that
$$(x\partial_{x}-(\beta+i))x^{\beta+j}=(j-i)x^{\beta+j}$$ If
$f=\sum\limits_{i=0}^{k}\alpha_{i}x^{\beta+i}\in M_{\alpha}^{\beta}$
where $\alpha_{k}\neq0$, then
$$\prod_{i=0}^{k-1}(x\partial_{x}-(\beta+i))f=\alpha_{k}k!x^{\beta+k}.$$
So some monomial $x^{\beta+k}\in A_{1}f$. Now use the formulas
\begin{equation}
\partial_{x}^{i}x^{\beta+k}=(\beta+k)...(\beta+k-i)x^{\beta+k-i}\neq 0
\end{equation} by assumption, and
\begin{equation}
x^{i}x^{\beta+k}=x^{\beta+k+i},
\end{equation} to show that all monomials $x^{\beta+i}\in A_{1}f$ for all $i\in \bZ$, and so $\rm{M_{\alpha}^{\beta}}\subset A_{1}f$. Since $f$ was an arbitrary element, this means that $\rm{M_{\alpha}^{\beta}}$ is simple. 
\end{proof}
The main result of this paper is  the following theorem that gives the number of decomposition factors when
$n=2$. Note that the different possibilities are distinguished by linear conditions on $\beta$.

\begin{thm}\label{MT} Assume that $n=2$. Let
$\beta=(\beta_{1},...,\beta_{m})\in\bC^{m}$ and $k$ be the number of
$\beta_{i}\in\bZ$. \\(i) If $k=m$, then
$\rm{c(M_{\alpha}^{\beta})=2m}.$\\(ii) If $k<m$ and\\

\qquad(a) $\sum_{i=1}^{m}\beta_{i}\in\bZ$, then
$\rm{c(M_{\alpha}^{\beta})=m+k-1}.$ \\

\qquad (b) $\sum_{i=1}^{m}\beta_{i}\in\bC\setminus\bZ$, then
$\rm{c(M_{\alpha}^{\beta})=k+1}.$
\end{thm}

As a corollary,  $M_{\alpha}^{\beta}$ is simple in exactly the following two cases: When $m\leq 2$ and $k=0$, or if $m\geq 3$ and $k=0$ and $\sum_{i=1}^{m}\beta_{i}\in\bC\setminus\bZ$.

\bigskip

It is sometimes said that illustrations of concrete calculations with D-modules are scarce; this has partly motivated this note.   D-modules on hyperplane configurations has been of interest to several authors, e g \cite{SST, KV, Walter}, not to mention many works on the corresponding equivalent category of sheaves.
Khoroshkin and Varchenko (see \cite{KV}) study a subcategory of holonomic D-modules with regular singularities along the stratification given by the intersections of the hyperplanes, and describe it in terms of quivers.  This category however does not include our modules.  By the Riemann-Hilbert correspondence and the known description of D-modules in terms of quivers, the problem studied in this note corresponds both to an assertion in a certain category of perverse sheaves and to an assertion on the number of decomposition factors of a quiver(see \cite{JEB2, B}, over a certain path algebra, and it would be illuminating to see a proof of our theorem using these tools.
Another technique that we do not explicitly use are Bernstein-Sato polynomials. What we indirectly construct,  corresponds more directly to a multi-dimensional version(see \cite{Sabbah}), but it may be mentioned that they have been calculated for hyperplane configurations in e g \cite{Walter}.

The organization of the paper is as follows. In section 2 we prove that $M_{\alpha}^{\beta}$ is holonomic  and give a lemma on external products that we will need and some preliminary results. In the next section we
consider the easy normal crossings case when $m\leq n$ and also the case---our starting point---
when all $\beta_{i}\in\bZ$, which corresponds to
$\rm{M_{\alpha}^{\beta}\cong\bC[x]_{\alpha}}.$ In the last case
there is a combinatorial description of the number of decomposition
factors. Then follows in section 4 the main part of the proof of Theorem \ref{MT}. The main idea of the argument is to study the annihilator of
$\alpha^{\beta}\in \rm{M_{\alpha}^{\beta}}$. 
In section 5 we conclude the proof of Theorem \ref{MT}, and also describe the support of the decomposition factors.

\section{Preliminaries}
\subsection{Some easy properties of $M_{\alpha}^{\beta}$ }
It is immediate that  $M_{\alpha}^{\beta}$ is holonomic, since it is the direct image of a connection on $X$. This is also easy to see directly, by copying the proof for the localization(which corresponds to the case $\beta=0$). The latter argument also provides an estimate of the multiplicity of the module, so we sketch it here.

\begin{proposition}\label{Prop12} (i) $\rm{M}_{\alpha}^{\beta}\cong \rm{M}_{\alpha}^{\gamma}$, if $\beta\equiv\gamma$ $(\rm{mod}\bZ^{m}$).\\
(ii) $\rm{M}_{\alpha}^{\beta}\cong\mathbb{C}[x]_{\alpha}$, if
$\beta\in\bZ^{m}$.\\
(iii) $M_{\alpha}^{\beta}$ is holonomic with multiplicity less than $(m+1)^n$.
\end{proposition}

\begin{proof}(ii) is a special case of (i). Suppose that $\beta=\gamma+\tau$, $\tau\in\bZ^{m}$.\\
 Define $\theta:\rm{M}_{\alpha}^{\beta}\longrightarrow \rm{M}_{\alpha}^{\gamma}$ by: $$\theta(\frac{p}{\alpha^{r}}\alpha^{\beta})=\frac{p}{\alpha^{r}}\alpha^{\tau}\alpha^{\gamma}$$
 Clearly this is a 1-1, onto map and it is an easy exercise to show that it is an $A_{n}$-module homomorphism. This proves (ii).
 
To prove holonomicity, recall that $m$ is the degree of $\alpha$ and for $k\geq 0$ set 
$$\Gamma_{i}=\{\frac{q}{\alpha^{k}}\alpha^{\beta}:q\in\bC[x],\rm{deg}q\leq(m+1)i\}.
$$  It is straightforward to prove that this filtration has the following properties.
\begin{itemize}
\item $ \Gamma_{i}\subset \Gamma_{j}$ if $i\leq j$.
\item $ \cup_{i\geq 0} \Gamma_{i}=\rm{M}_{\alpha}^{\beta}$.
\item $x_i\Gamma_j\subset \Gamma_{j+1}$, $\partial_{x_i}\Gamma_j\subset \Gamma_{j+1}$.
\end{itemize}
Hence for every fixed $i$ the submodule $A_n\Gamma_i$ is holonomic. The dimension of $\Gamma_{j}$ cannot exceed the dimension of the vector space of polynomials of degree $(m+1)j$ and so
 $$\rm{dim}_{\bC}(A_n\Gamma_i)\cap \Gamma_j\leq\left(\begin{array}{ccc}(m+1)j+n\\
 n
 \end{array}\right)\leq\frac{(m+1)^{n}j^{n}}{n!}+cj^{n-1}$$ for large enough values of $j$ and some constant c.  By (\cite{CSC}, Ch.10 Lemma 3.1), $A_n\Gamma_i$ is a holonomic module of multiplicity less than or equal to $(m+1)^{n}$. By this fix upper bound on the multiplicity, the additivity of multiplicity, also $\rm{M}_{\alpha}^{\beta}$ has to be holonomic(and the filtration is a good one, by the way).
\end{proof}

\subsection{External product of modules} Let  A, B be $\bC$-algebras. Suppose that M  is a left A-module and N is a left B-module. Then the $\bC$-vector space $M\otimes_{\bC}N$ is an $A\otimes_{\bC}B$-module denoted by $M\widehat{\otimes}N$ and called the {\it{external product}} of $M$ and $N$. The action of $a\otimes b\in A\widehat{\otimes}B $ on $u\otimes v\in M\otimes_{K}N$ is given by the formula $(a\otimes b)(u\otimes v)=au\otimes bv.$ If $A=A_n$ and $B=A_m$, then $M\widehat{\otimes}N$ is in this way an $A_{n+m}$-module(for this see \cite{CSC}). We will use that the external product of two simple Weyl algebra modules is simple and include a proof for convenience, shown to us by Rolf K\"allstr\"om. 
We need the following well-known result \cite[2.6]{DJ1}.
\begin{lemma}\label{Lem14} If M is a simple $A_{n}$-module then $\rm{End_{A_{n}}M}=\mathbb{C}.$
\end{lemma}

\begin{lemma}\label{Prop14} Let M be a simple $A_{n}$-module and N be a simple $A_{m}$-module. Then $M\widehat{\otimes}N$ is a simple $A_{m+n}$-module.
\end{lemma}
\begin{proof} Let $0\neq f=\sum_{i=1}^{k}m_{i}\otimes n_{i}\in M\widehat{\otimes}N$ be an arbitrary element. We want to show that $A_{m+n}f=M\widehat{\otimes}N$. We will make an induction on $k$.

\subsection*{Step I}Let $f=m_{1}\otimes n_{1}, m_{1}\neq 0, \ n_{1}\neq 0$. 
Let  $g=\sum_{i=1}^{k}r_{i}\otimes s_{i}\in M\widehat{\otimes}N$ be an arbitrary element.
We know that $A_{n}m_{1}=M$ and $A_{m}n_{1}=N$ and hence there are $a_i\in A_n$ and $b_i\in A_m$ such that
$r_i=a_im_1$ and $s_i=b_in_1$, for $i=1,...,k$. Then $g=(\sum_{i=1}^{k}a_{i}\otimes b_{i})(m_1\otimes n_1)$, and hence $A_{m+n}f=M\widehat{\otimes}N$.

\subsection*{Step II} Now let $f=\sum\limits_{i=1}^{k}m_{i}\otimes n_{i}, m_{i}\in M , n_{i}\in N$ where $k\geq 2$ and  $n_i,\ i=1,...,k$ are linearily independent (over $\bC$). Let $J_{i}:={\rm{Ann}}(m_{i})$. 
Assume first that there are  $1\leq i,j\leq k$ and $a$ such that $a\in J_i\setminus J_j$. Then $h:=(a\otimes1)f=\sum\limits_{i=1}^{k}m_{i}\otimes n_{i}\neq 0$, is a sum of fewer terms than $f$ and by induction $A_{m+n}h=M\widehat{\otimes}N$. This implies that $A_{m+n}f=M\widehat{\otimes}N$.
Otherwise $J_i=J_j$ for all $i,j$. Then we have an isomorphism $\phi_{1}:A_{n}/J_{1}\longrightarrow
M$ defined by 
 $ a+J_{1}\longmapsto am_{1}$ and a similar isomorphism $\phi_{2}:A_{n}/J_{2}\longrightarrow M$.
 By Lemma~\ref{Lem14}, $\eta(m):=\phi_{2}o\phi_{1}^{-1}(m)=\alpha m$ for some $\alpha\in\mathbb{C}$ and all $m\in M$. This implies $\eta(m_{1})=\alpha m_{1}=m_{2}$ and hence  
 $$f=m_{1}\otimes (n_{1}+\alpha n_2)+m_{3}\otimes n_{3}+...+m_{k}\otimes n_{k}.$$ 
 Thus, again by induction $A_{m+n}f=M\widehat{\otimes}N$. This finishes the proof.

 \end{proof}

\subsection{Decomposition factors}\label{ss-16}
Let R be a ring and M be an R-module. If $0=M_{0}\subset
M_{1}\subset\dots M_{r}=M$ is a composition series of M, then define the
set  of decomposition factors as $$\rm{DF}(M):=\{ M_{i}/M_{i-1}\}_{i=1}^{r},$$  
and let $c(M)$ be the number of decomposition factors. We will use the following
Lemma on the decomposition factors of R-modules.

\begin{lemma}\label{Prop16} Let M be an R-module.\\
 Let N be a submodule of M.
 Then\\
 \qquad $ (i)$ $\rm{DF}(M)=\rm{DF}(N)\cup \rm{DF}(M/N)$,\\
\qquad $(ii)$ $c(M)=c(N)+c(M/N).$
\end{lemma}

\begin{cor}\label{Cor11} Let $0=M_{0}\subset M_{1}\subset...\subset M_{k}=M$ be a composition series of an $A_{n}$ module M and $0=N_{0}\subset N_{1}\subset...\subset N_{l}=N$ be a composition series of an $A_{m}$-module N. Then $$\rm{DF}(M\widehat{\otimes}N)=\{M_{i}/M_{i-1}\widehat{\otimes}N_{j}/N_{j-1}\}_{i=1,j=1}^{k,l}$$ and hence $c(M\widehat{\otimes}N)=c(M)c(N)$.
\end{cor}
\begin{proof} It suffices to note that $M_{i}/M_{i-1}\widehat{\otimes} N_j/N_{j}$ is simple by Lemma~\ref{Prop14} and that taking exterior tensor product is a flat functor.
\end{proof}

\section{Examples}
\subsection{Normal Crossings}\label{s-4}
 Consider first the case when $m\leq n$ and recall that we have a blanket assumption that $\alpha_{1},...,\alpha_{m}$ are linearly independent. After a base change the linear forms may be taken to be  $\alpha_{1}=x_{1},...,\alpha_{m}=x_{m}$.  Recall that, in section~\ref{s-1} of this paper, we  considered the case $m=n=1$. Now using that we are going to treat the general case.

\begin{proposition}\label{Thm42}Let $M_{\alpha}^{\beta}=\bC[x]_{\alpha}\alpha^{\beta}, \alpha=x_{1}...x_{m}$ and $m\leq n$.
Assume that $\beta_{1},\dots,\beta_{k}\in\bZ$ and $\beta_{k+1},\dots,\beta_{m}\in\bC\setminus\bZ$. Then $c(M_{\alpha}^{\beta})=2^{k}$. In particular $M_{\alpha}^{\beta}$ is simple if and only if $\beta_{1},\dots,\beta_{m}\in\bC\setminus\bZ$.
\end{proposition}
\begin{proof}
 The multiplication map induces an isomorphism:
$$\bC[x]_{x_{1}...x_{m}}x_{1}^{\beta_{1}}...x_{m}^{\beta_{m}}\cong 
(\widehat{\otimes}_{i=1}^m\bC[x_{i}]_{x_{i}}x_{i}^{\beta_{i}})\widehat{\otimes}(\widehat{\otimes}_{i=m+1}^n\bC[x_{i}]).$$
Now $c(\bC[x_{i}]_{x_{i}}x_{i}^{\beta_{i}})$ is 2 or 1 according as
$\beta_{i}$ is an integer or not. The result then follows by Corollary \ref{Cor11}.
\end{proof}

\subsection{The case of $\bC[x]_\alpha$} \label{integer}By Proposition \ref{Prop12} this corresponds to $\beta\in \bZ^m$. In this case there is a known combinatorial description of $c(M_\alpha^\beta)$ in terms of certain linearily independent subsets of $\Theta:=\{\alpha_1,...,\alpha_m\}$. We will give a short description(for details see \cite{CDCP2}.)

Recall first that to each closed embedding $i:Y\to X$ between smooth varieties, such that $Y$ has codimension $d$, there is, by Kashiwaras theorem, associated a simple $\D_Y$-module $L_{Y/X}=i_*\Oc_Y=H^d_Y(\Oc_X)$. When $X=\bC^n$ and $0=Y=V(x_1,x_2,...,x_n)$ is a point, $L_{0/X}$ may be described as 
$$
L_{0/X}=\frac{\bC[x_1,...,x_n]_{x_1....x_n}}{\Sigma_{i=1}^n \bC[x_1,...,x_n]_{x_1...\hat x_i....x_n}}\cong \bigoplus_{\beta\in \bZ^n_-}\bC x^\beta ,
$$
where $\hat x_i$ signifies that the corresponding variable is not present in the expression, and $\bZ_-$ is the set of strictly negative integers. More generally, when $Y=V(x_1,x_2,...,x_m)$,  $L_{Y/X}$ may be described as the exterior tensor product of an $A_m$-module and an $A_{n-m}$-module:
$$L_{Y/X}\cong L_{0/\bC^m}\widehat {\otimes}\ \bC[x_{m+1},....,x_n]\cong \bigoplus_{\beta_1\leq -1,...,\beta_m\leq -1}\bC x^\beta.
$$ 
Note that this corresponds to choosing a vector space splitting $C^n=Y\times Z$, and describing
 $L_{Y/X}\cong L_{0/Z}\widehat {\otimes}\ L_{Y/Y}$, where $\bC[x_{m+1},....,x_n]\ =\Oc _Y=L_{Y/Y}$.
 We will now fix terminology associated to a hyperplane arrangement..
 	
\begin{definition}Let $R:=\bC[x_1,...,x_n]$. Furthermore let $\Theta=\{\alpha_1,...,\alpha_m\}$ be forms,  and if $S\subset \theta $ let $\alpha_S=\prod_{\alpha\in S}\alpha$ and $R_S:=R_{\alpha_S}$. The subset $S\subset \Theta$ defines the linear subspace $H_S$ where all the forms in $S$ vanish, i.e. the flat $H_S=V(\alpha\in S)$. 
Furthermore let $i_S: H_S\to \bC^n$ be the inclusion and $L_S={i_S}_*\Oc_{H_S}$ the corresponding simple $A_n$-module as above.
\end{definition}

The $L_S$ are the modules that occur in a decomposition series of $\bC[x]_{\alpha}$. Note that $L_{S_1}$ and  $L_{S_2}$ are isomorphic exactly when $H_{S_1}=H_{S_2}$. First we need the {\it polar filtration.}

\begin{definition}
Define $$R_i:=\Sigma_{\vert S\vert=i} R_S\subset \bC[x]_{\alpha} $$
(and $R_{-1}:=0)$).
Clearly $R_S$ is an $A_n$-submodule of $\bC[x]_{\alpha}$, and $\bC [x]=R_0\subset R_1\subset ....$, gives an $A_n$-module filtration of $\bC[x]_{\alpha}$.  
\end{definition}

\begin{thm}[\cite{CDCP2}] (i) dim supp $R_{i}/R_{i-1}= n-i$, for $i=0,...,n$.\\
(ii) $R_{i}/R_{i-1}$ is a semi-simple $A_n$-module.\\
(iii)$R_{i}/R_{i-1} \cong \oplus_{S} L_S$, where the sum is taken over a certain set, of all so-called {\it no-broken}(see \cite{CDCP2}) linearily independent subsets of $\Theta$ of cardinality $i$. The image of $L_S$ in $R_i/R_{i-1}$ is generated by $\alpha_S^{-1}.$
\end{thm}

\begin{remark}  The polar filtration is the ordinary dimension filtration of a module over a noetherian commutative ring, only that the commutative ring in this case is the subalgebra $\bC[\partial]$ of constant differential operators. Since the module is holonomic, its  characteristic varieties have dimension $n$ over the symbol algebra. Hence this filtration by increasing dimension induces a filtration over $\bC[x]$ where the graded pieces have decreasing dimension.
\end{remark}

 If $n=2$ we have the following sequence of $A_{2}$-modules $$0\rightarrow R_{0}(=\bC[x,y])\subset R_{1}\subset R_{2}=\bC [x]_\alpha.$$ If the forms are $\Theta=\{\alpha_1,....,\alpha_m\}$, then the set of no-broken linearily independent subsets is $$\{\emptyset,\{\alpha_1\},...,\{\alpha_m\}, \{\alpha_1,\alpha_2\},...,\{\alpha_1,\alpha_m\}\}.$$
 So in the plane case $\bC [x]_\alpha$ decomposes in the following way.
 \begin{cor}\label{integer.cor} (Plane case) (i) There is one unique simple factor $\bC[x,y]$ with support equal to $\bC^2$.\\
 (ii) The simple factors with support on lines are given by the $m$ non-isomorphic modules in the direct sum decomposition
 $$R_{1}/R_0=\bigoplus_{i=1}^m L_{\alpha_i}$$\\
( iii) There are $m-1$ isomorphic simple decomposition factors with support at the origin:
\begin{equation}
\label{eq: L}
R_{2}/R_1=\bigoplus_{i=2}^m L_{\{\alpha_1,\alpha_i\}}\end{equation}\\
(iv) $c(\bC[x]_{\alpha})=2m$. 
  \end{cor}

For example, the image of the decomposition factor $L_S$ corresponding to $S=\{\alpha_1,\alpha_2\}$ in $R_2/R_1$ is just $$Im L_{\{\alpha_1,\alpha_2\}}=\bigoplus_{\beta_1\leq -1,\beta_2\leq -1}\bC \alpha_1^{\beta_1} \alpha_2^{\beta_2}.$$

We will need to know the image of $L_{\alpha_i}$. It may be described as follows. 
Choose for each $\alpha_i$ either $\alpha^c_i=x$ or $\alpha^c_i=y$, so that $\alpha_i, \alpha^c_i$ are linearily independent.
Then the image of $L_{\alpha_i}$ in $R_1/R_0$ is the vectorspace 
\begin{equation}
\label{eq:L2}
Im L_{\alpha_i}=\bigoplus_{k\geq 0,\ l\leq -1}\bC { (\alpha^c_i)}^k\alpha_i^l.
\end{equation}

\section{Proof of Theorem \ref{MT} when all
$\beta_{i}\in\bC\setminus\bZ$}\label{s-6} 
 This section consists of the proof of a special case of Theorem \ref{MT}, to which the general case may be reduced.
By a change of coordinates assume in the sequel that
 $$M_{\alpha}^{\beta}=\mathbb{C}[x,y]_{\alpha}\alpha^{\beta},$$ where
 $\alpha=xy\prod_{i=3}^{m}(c_ix+y)$, and $c_{i}\neq c_{j}\neq 0$ for $i\neq j$. We will continue to call the forms $\alpha_1=x,\alpha_2=y,....$, and extend the definition of $c_i=\partial_x(\alpha_i)$(so $c_0=1,\ c_1=0)$. Note that the case when $m\leq 2$ was treated in Proposition \ref{Thm42}. With no $\beta_i$ integers, it turns out that whether  $\vert\beta\vert:=\sum_{i=1}^{m}\beta_{i}$ is an integer or not,  determines whether the module is simple.
 
  \begin{thm}\label{Thm61}  
 Assume that $\beta_{i}\in\bC\setminus\bZ$, $i=1,...,m. $\\(ii) If $\vert\beta\vert:=\sum_{i=1}^{m}\beta_{i}\in\bZ$, then  $c(M_{\alpha}^{\beta})=m-1$.
\\ (ii) If $\vert\beta\vert\in \bC\setminus\bZ$ , then
$c(M_{\alpha}^{\beta})=1$ and $M_{\alpha}^{\beta}$ is simple.
 \end{thm}
 Studying the annihilator of $\alpha^{\beta}$ is the crucial ingredient of the proof.

\subsection{Normal form algoritm}\label{ss-62}

The Euler derivation  $P=P(\beta)=x\partial_{x}+y\partial_{y}-\vert\beta\vert$ belongs to the annihilator of $\alpha^{\beta}$. Another annihilator is found by considering the action of $\partial_y$:
$$
(\partial_y-\sum_{i=2}^{m}\frac {\beta_i}{\alpha_i})\alpha^{\beta}=0
$$
 and clearing denominators. This gives 
 $$
 Q=Q(\beta):=(\prod_{j=2}^{m}\alpha_{j})\partial_{y}-\sum_{i=2}^{m} \beta_{i}\prod_{j=2,j\neq i}^{m}\alpha_{j}.
 $$
 (the similarily constructed annihilator involving $\partial_x$ is contained in the ideal generated by $P$ and $Q$).  We will next describe a normal form for elements modulo $P$ and $Q$ .
For 
this use the \textit{graded reverse lexicographic order}, ordering the variables by $y>x>\partial_{x}>\partial_{y}$, and letting 
$$y^{i_1}x^{j_1}\partial_{x}^{k_1}\partial_{y}^{l_1}>y^{i_2}x^{j_2}\partial_{x}^{k_2}\partial_{y}^{l_2}$$ if $$i_1+j_1+k_1+l_1>i_2+j_2+k_2+l_2$$ or $$i_1+j_1+k_1+l_1=i_2+j_2+k_2+l_2  $$ and the last non-zero coordinate of 
$(i_1-i_2,j_1-j_2,k_1-k_2,l_1-l_2)$ is negative. The $normalform$ algorithm, see \cite[Chapter 1]{SST} and \cite[Chapter 2]{CLO}, with respect to $\{ P,Q\}$, inputs an element $F$ of the Weyl algebra  and outputs an element $R$ such that there exist $S_1$ and $S_2$ in the Weyl algebra with $F=S_1P+S_2  Q+R$ and where  the initial term of $R$ is not divisible by the initial terms of $P$ and $Q$. Since the initial term of $P$ is $\underline{x\partial_{x}}$ and the initial term of $Q$ is $\underline{y^{m-1}\partial_{y}}$, it follows that 
$$A_{2}=A_{2}P+A_{2}Q+N$$
 where $$N=\oplus_{(i,j,k,l)\in M}\bC y^{i}x^j \partial_{x}^{k}\partial_{y}^{l}$$ 
 and $M\subset\bZ_{\geq0}^{4}$, is the set 
\begin{equation}
\label{eq:normalform}
M=\{(i,j,k,l):\ jk=0,\ l\neq 0\implies i\leq m-2\}.
\end{equation}
 We will use this in a slightly different form. Give the the variables $x,y\in A_2$ weight $1$, and the  derivations $\partial_x,\partial_y$ weight $-1$. Then both $P$ and $Q$ are homogeneous elements. Denote by $(A_2)_0\subset A_2$ the space of homogeneous differential operators of weight $0$.
 
 \begin{lemma}
 \label{normalform}
 $$
 (A_2)_0\subset N_0+A_2P+A_2Q
 $$
  where $N_0$ is the vector space 
  \begin{equation}L=\bigoplus_{1\leq k}\bC (y\partial_{x})^{k}\oplus\bigoplus_{1\leq k,l; \ k+ l\leq m-2}\bC (y\partial_{x})^{k}(y\partial_{y})^{l}\oplus\bigoplus_{l\leq m-2;\  0\leq k}\bC (y\partial_{y})^{l}(x\partial_{y})^{k}.
  \end{equation}
\end{lemma}
\begin{proof}Since $P$ and $Q$ are homogeneous elements $$(A_{2})_0=(A_{2}P)_0+(A_{2}Q)_0+N_0$$
 where $N_0=N\cap (A_2)_0$ differs from $N$, by the added condition that $i+j-k-l=0$ in (\ref{eq:normalform}).  Using that $(A_2)_0$ is generated by monomials in $x\partial_{y}, x\partial_{x}, y\partial_{y}$ and $y\partial_{x}$,  and changing the presentation of the monomials in $N_0$ gives the description of the monomials in the Lemma.
 \end{proof} 
 
 \subsection{ The annihilator of $\alpha^{\beta}$}
   
  Now we can get some information on the annihilator of $\alpha^{\beta}$. 
  
\begin{lemma}\label{Prop61}Assume that $\beta_{i}\in\bC\setminus\bZ,\  i=1,...,m.$
\begin{enumerate} 
\item Let $ \rm{Ann}=\rm{Ann}_{A_{2}}\alpha^{{\beta}}$. Then $ A_2P+A_2Q\subset \rm{Ann}$.
\item The homogeneous part  of $ \rm{Ann}$
$$ \rm{Ann}_0\subset A_2P+A_2Q.$$
\end{enumerate}
\end{lemma}
\begin{proof} Note that (i) does not need the hypothesis and is a routine verification. We will have use for the {\it valuation} with respect to the form $L=\alpha_i,\ i=1,2,...,m$.
 If
  $m\in M_{\alpha}^{\beta}$, define $O_{L}(m)$ to be the greatest $k\in \bZ$ such that $$m=r(x,y)L^{k}\alpha^{\beta}$$ where $r(x,y)\in \bC (x,y)$ is a quotient of two polynomials that are not divisible by $L$(using that the ring of polynomials is an UFD). Furthermore define 
 $$V_{L, k}=\{ m\in M_\alpha^\beta : O_L(m)\geq k\}.
 $$
 Clearly $...\supset V_{L, -1}\supset V_{L,0}\supset V_{L,1}\subset ....$
 To avoid confusion, we emphasize that when we below speak of the {\it order} of a differential operator in $A_2$ and the {\it order filtration} of $A_2$, it will be in the ordinary sense of differential operators. 
  Note that if $S\in A_2$ has order $s$, then
\begin{equation}
\label{Order0}SV_{L,k}\subset V_{L,k-s},
\end{equation} 
and that 
\begin{equation}
\label{eq:orderx}
\partial_y V_{x,k}\subset V_{x,k}\ \ {\rm and} \ \ \partial_x V_{y,k}\subset V_{y,k}.  
\end{equation}

After these preliminaries, suppose now that $B\in  \rm{Ann}$ is homogeneous. Using the previous Lemma \ref{normalform},  $B= U+S_1P+S_2Q$, where $U\in L$ and $S_i\in A_2,\  i=1,2$. Hence, by (1), $U\in  \rm{Ann}$. We will show that $U$ must be $0$.

Let $r$ be the order of $U$. If $r\geq m-1$, then the leading term, with respect to the order filtration, of $U$ will look like $$
\gamma(y\partial_x)^r+\sum_{l\leq m-2 ,l+n=r,n}\alpha_{l,n}(y\partial_{y})^{l}(x\partial_{y})^{n}=\gamma y^r\partial_x^r+ R(x,y)x^{r-(m-2)}\partial_{y}^{r},
$$ where the degree of $R(x,y)$ is equal to $m-2$. 
Consider first valuation with respect to $x$. Evaluate modulo $V_{x,-r+1}$, using  (\ref{eq:orderx}): 
$$
0=U\alpha^\beta=\beta_1(\beta_1-1)...(\beta_1-r+1)\gamma \frac{y^r}{x^r},
$$ and hence $\gamma=0
$
(since $\beta_1\notin \bZ$).
Next let $L=\alpha_i,\ i=2,...,m$. Evaluate again modulo $V_{L,-r+1}$, using  (\ref{Order0}): 
$$
0=U\alpha^\beta= R(x,y)\beta_i(\beta_i-1)...(\beta_i-r+1)x^{r-(m-2)}L^{-r}+ V_{L, -r+1}.
$$ Since $\beta_i\notin \bZ$,  $L$ divides  $R(x,y)$. Repeating this for all $\alpha_i\ i=2,...,m$ implies that $\alpha_2L\alpha_{3}...\alpha_{m}$ divides
$R(x,y)$. This contradicts that $\rm{deg}_{y}R(x,y)\leq m-2.$ Hence
$R(x,y)=0$, and so we are reduced to the second case, when $r\leq m-2$. Clearly we may assume that $r>0$, since otherwise $U$ would be a constant annihilating $\alpha^\beta$ and so $U=0$. 
Then the leading term of $U$ will look like $$
\sum_{k=0}^r\alpha_{k}(y\partial_{y})^{k}(y\partial_{x})^{r-k}.$$ 
Consider now again the different valuations. With respect to $L=x$, and in view of (\ref{eq:orderx}), we have that
$$ 0=U\alpha^\beta=\frac{\alpha_{0}\beta_1(\beta_1-1)...(\beta_1-r+1)y^r}{x^r}\ \ ({\rm Mod} \ V_{x,-r+1}).
$$
This implies that $\alpha_{0}=0$. Now if $L=\alpha_i=c_ix+y, \ i=3,...m$,
$$ U\alpha^\beta=(\sum_{k=0}^r\alpha_{k}c_i^{r-k})\frac{\beta_i(\beta_i-1)...(\beta_i-r+1)y^r}{L^r}+V_{L,-r+1}.$$
 This gives $m-2$ equations  $p(c_i)=\sum_{k=1}^r\alpha_{k}c_i^{r-k}=0$, where $p(x)=\sum_{k=1}^r\alpha_{k}x^{r-k}$, and since the degree of $p$ is  $r-1\leq m-3$, we have $p(x)=0$. This finally implies that $U=0$, and finishes the proof of the Lemma.
 \end{proof}

\subsection{The structure of $A_2/(A_2P+A_2Q)$}\label{ss-63}
We have a surjective map 
$$
A_2/(A_2P+A_2Q)\twoheadrightarrow A_2/{\rm Ann}\subset M_\alpha^\beta,
$$
and we will find the decomposition factors of $M_\alpha^\beta$ by analyzing the first module.
\begin{lemma}
\label{Lem61}Let $\tilde{\beta}=\beta+N,$ where $N\in\bZ^{m}$ and $\alpha^{\tilde{\beta}}=\alpha^N\alpha^{{\beta}}\in M_{\alpha}^{\beta}$. Then $$J=J(\tilde{\beta}):=A_{2}x+A_2P(\tilde{\beta})+A_2Q(\tilde{\beta})=A_{2}x+A_{2}(y\partial_{y}-(\vert\tilde{\beta}\vert+1))+A_{2}(y^{m-2}).$$
\end{lemma}
\begin{proof} Recall that $P(\tilde{\beta})=x\partial_{x}+y\partial_{y}-\vert\tilde{\beta}\vert=\partial_{x}x+y\partial_{y}-(\vert\tilde{\beta}\vert+1)$. This shows that $y\partial_{y}-(\vert\tilde{\beta}\vert+1)\in J$.
Secondly, 
$$
Q(\tilde{\beta})=(\prod_{j=2}^{m}\alpha_{j})\partial_{y}-\sum_{i=2}^{m}\tilde \beta_{i}\prod_{j=2,j\neq i}^{m}\alpha_{j}=
Gx+(y^{m-1}\partial_{y}-\sum_{i=2}^{m}\tilde{\beta_{i}}y^{m-2})
$$
 for some $G\in A_{2}$.
  Hence $$J=A_{2}x+A_{2}(y\partial_{y}-(\vert\tilde{\beta}\vert+1))+A_{2}(y^{m-1}\partial_{y}-\sum_{i=2}^{m}\tilde{\beta_{i}}y^{m-2}).$$But $y^{m-1}\partial_{y}-\sum_{i=2}^{m}\tilde{\beta_{i}}y^{m-2}-y^{m-2}(y\partial_{y}-(\vert\tilde{\beta}\vert+1))=(\tilde{\beta_{1}}+1)y^{m-2}\in J$. Since $\tilde{\beta_{1}}+1\neq 0$, by assumption, then $y^{m-2}\in J$, and $$J=A_{2}x+A_{2}(y\partial_{y}-(\vert\tilde{\beta}\vert+1))+A_{2}(y^{m-2}).$$
\end{proof}
\begin{lemma}
\label{Lem62}Let $A_{1}=\bC\langle y,\partial_{y}\rangle$. Let $J=A_{1}(y\partial_{y}-\gamma)+A_{1}y^{k}$ for $k\geq0$. Then we have the following. \\(i) If $\gamma\notin\{-1,...,-k\},$ then $J=A_{1}$.\\ (ii) If $-k\leq\gamma\leq-1$, then $J=A_{1}(y\partial_{y}-\gamma)+A_{1}y^{\vert\gamma\vert}$. Further more $$A_{1}/J \cong \bC[y]_{y}/\bC[y]$$ and is hence simple.
\end{lemma}
\begin{proof}(i) If $\gamma\notin\{-1,...,-k\},$ then $j+\gamma\neq0$, for $j\in\{1,...,k\}$. Now,
$$\partial_{y}y^{k}-y^{k-1}(y\partial_{y}-\gamma)=(k+\gamma)y^{k-1}\in J.$$ Since $k+\gamma\neq0$, then $y^{k-1}\in J.$ Iterating we find that $1\in J$, since by assumption $k+\gamma\neq0,k-1+\gamma\neq 0,...,1+\gamma\neq0,$ and hence $J=A_{1}.$\vspace{2mm}\\ 
(ii) If $-k\leq\gamma\leq-1$, the same argument still gives that  $J=A_{1}(y\partial_{y}-\gamma)+A_{1}y^{\vert\gamma\vert}$.
Let $\theta :A_{1}\longrightarrow\bC[y]_{y}/\bC[y]$ be the map
defined by $\theta(P)=P(\bar{y}^{\gamma})$. Clearly $J\subset
\rm{Ker}\theta$ and $\theta$ is surjective. Now, \begin{equation*}
 A_{1}=J+\oplus_{i\geq0}\bC\partial_{y}^{i}\oplus\oplus_{j=1}^{\vert\gamma\vert-1}\bC y^{j},
\end{equation*} so $J= \rm{Ker}\theta$. This
concludes the proof.
\end{proof}
\begin{lemma}
\label{Lem63}$$ A_{2}/(A_{2}x+A_2P(\tilde\beta)+A_2Q(\tilde\beta))\cong A_{2}/(A_{2}x+{\rm Ann}(\tilde\beta))\cong A_{2}\alpha^{\tilde{\beta}}/A_{2}x\alpha^{\tilde{\beta}}$$ 
is a non-trivial simple $A_{2}$-module if and only if $-(m-2)\leq\vert\tilde{\beta}\vert+1\leq-1$ and zero otherwise.
\end{lemma}
\begin{proof}The last isomorphism is obvious. By Lemma~\ref{Lem61}, $$A_{2}/(A_{2}x+A_2P(\tilde\beta)+A_2Q(\tilde\beta))\cong A_{2}/(A_{2}x+A_{2}(y\partial_{y}-(\vert\tilde{\beta}\vert+1))+A_{2}(y^{m-2})).$$ So this module is the external product 
$$(\bC\langle x,\partial_{x}\rangle/\bC\langle x,\partial_{x}\rangle x)\widehat{\otimes}(\bC\langle y,\partial_{y}\rangle/\bC\langle y,\partial_{y}\rangle\langle y\partial_{y}-(\vert\tilde{\beta}\vert+1),y^{m-2}\rangle).$$ Hence the result on simplicity follows by Lemma~\ref{Lem62} (ii) and Lemma~\ref{Prop14}. It remains to prove the first isomorphism. There is a canonical surjection 
$$
\theta: A_2/(A_2x+A_2P(\tilde\beta)+A_2Q(\tilde\beta))\to A_2/{A_2x+{\rm Ann}(\tilde\beta)}.
$$
If the second module is zero, then there is $B\in A_2,\ C\in {\rm Ann}(\tilde\beta)$ such that $1=Bx+C\iff C=1-Bx\in {\rm Ann}(\tilde\beta)$. Take the homogeneous weight $0$ part: $1=B_{-1}x+C_0$. Now, by Lemma \ref{normalform}, $C_0\in A_2P(\tilde\beta)+A_2Q(\tilde\beta)$ and so $1\in A_2x+A_2P(\tilde\beta)+A_2Q(\tilde\beta)$. Hence the first module is zero. Since it is always simple, this proves that the surjection $\theta $ is always an isomorphism and finishes the proof of the Lemma.

\end{proof}
\begin{lemma}\label{Lem64}
(i) There exists $N_{1}\in \bN^m$ such that $\alpha^{\beta-N_{1}}$ generates $M_{\alpha}^{\beta}$.\\
(ii) There exists $N_{2}\in \bN^m$such that
$A_{2}\alpha^{\beta+N_{3}}$ is a simple submodule if $N_{3}\in
N_{2}+\bN^{m}$.
\end{lemma}
\begin{proof} (i) follows directly from the fact that $M=M_{\alpha}^{\beta}$ is a holonomic module and hence Noetherian, see \cite{CSC}. Let $e=(1,...1)\in \bN^m$. By the Noetherian property the ascending sequence of submodules $A_2\alpha^{\beta-ne}, \  n=1,2,...$ stabilizes at some $n=n_1$. But, for large enough $n$ any element in  $M$ is contained in $A_2\alpha^{\beta-ne}$. This holds in particular for a finite set of generators of $M$, and so $M$ is generated by $\alpha^{\beta-n_1e}$. 

For (ii), note similarily that the descending sequence of submodules $A_2\alpha^{\beta+ne}, \  n=1,2,...$ stabilizes if $n$ is larger than some $n_2$. As a further consequence of holonomicity, $M$ contains a simple submodule $M_1$. If $f$ is a generator of $M_1$, then also $\alpha^{ne}f,\ n\geq 0$ is a generator, and if we then choose $n$ big enough we may assume that $f\in A_2\alpha^{\beta+n_2e}$. Hence $M_1\subset A_2\alpha^{\beta+n_2e}$. Furthermore, all decomposition factors of $M$ have support on hyperplane intersections, and $M$ has rank 1 as a $\bC(x,y)-$module, so any element  in $ A_2\alpha^{\beta+n_2e}/M_1$ is annihilated by a large enough power of $\alpha$. In particular, there is $n_4\geq n_2$ such that $\alpha^{\beta+n_4e}\in M_1$, and it follows that $A_2\alpha^{\beta+n_2e}=A_2\alpha^{\beta+n_4e} =M_1$.
\end{proof}

\subsection{Proof of Theorem~\ref{Thm61}} We are now  in a position to prove
Theorem~\ref{Thm61}.   Assume first that $\vert{\tilde{\beta}}\vert \in\bC\setminus\bZ$. By Lemma~\ref{Lem63}, 
then $A_{2}\alpha^{\tilde{\beta}}/A_{2}x\alpha^{\tilde{\beta}}=0$ and hence $A_{2}\alpha^{\tilde{\beta}}=A_{2}x\alpha^{\tilde{\beta}}=...=A_{2}x^{r}\alpha^{\tilde{\beta_{1}}}$
 for any $r\in\bZ$ and this is true exchanging $x$ for any of the other linear forms.
 Then by Lemma~\ref{Lem64}, $A_{2}\alpha^{\beta+N_{3}}$=
 $A_{2}\alpha^{\beta+N_{1}}$=$M_{\alpha}^{\beta}$.
 Therefore $M_{\alpha}^{\beta}$ is simple.

Assume then that $\vert{\tilde{\beta}}\vert \in\bZ$. Again using the notation of the preceding Lemma, put
$\tilde{\beta}=\beta+N_{1}$, and consider $A_{2}\alpha^{\beta+N_{1}}$. Since, if $\alpha^{\beta+N}$ generates
$M_{\alpha}^{\beta}$, also $x^{-n}\alpha^{\beta+N}$ generates if
$n\geq0$, we may assume $\vert\tilde{\beta}\vert\leq-(m-1)$. By
Lemma~\ref{Lem63}, if $\vert\tilde{\beta}\vert$ is not one of
$-(m-1),...,-2$, we have that
$A_{2}\tilde{\beta}/A_{2}x\tilde{\beta}=0$. Hence
$A_{2}\alpha^{\tilde{\beta}}=A_{2}x\alpha^{\tilde{\beta}}=...=A_{2}\alpha^{\tilde{\beta_{1}}},$
where $\alpha^{\tilde{\beta_{1}}}=x^{r}\alpha^{\tilde{\beta}}$ such
that $\vert\tilde{\beta_{1}}\vert=-(m-1).$ Then by Lemma~\ref{Lem63}
$$A_{2}\alpha^{\tilde{\beta_{1}}}\supset
A_{2}x\alpha^{\tilde{\beta_{1}}}\supset...\supset
A_{2}x^{m-2}\alpha^{\tilde{\beta_{1}}},$$is a chain of strict
submodules such that each quotient is simple and has support at
$(0,0)$. The last submodule,
$A_{2}x^{m-2}\alpha^{\tilde{\beta_{1}}}$, has the property (by again
applying Lemma~\ref{Lem63} to each $\alpha_i$),
 that it equals
$A_{2}\alpha^{N}x^{m-2}\alpha^{\beta}$ for all $N\in\bN^{m}$, and
hence by Lemma~\ref{Lem64} is simple. \\ 
Hence  $M_{\alpha}^{\beta}$
has $m-2$ decomposition factors with support at the origin, and one
with support on $\bC^{2}$. This concludes the
 proof.

 \section{ The general plane case}
In this section, we are going to consider the remaining case when $1\leq k<m$ of the
$\beta's$
 are in $\bZ$ and the rest are in $\bC\setminus\bZ$. 
 Let then $\beta_{1},...,\beta_k \in\bZ$,  so they may be taken to be $0$, by Lemma \ref{Prop12}. Define
 $\tilde\beta:=(\beta_{k+1},...,\beta_{m})\in(\bC\setminus\bZ)^{m-k}$, and set $\tilde \alpha=\alpha_{k+1}...\alpha_{m}$.  We then want to study the module 
 $$
 M_\alpha^\beta=\bC[x,y]_{\alpha}{\tilde\alpha}^{\tilde \beta}.$$
 First, by section \ref{integer}
 the module $\bC[x,y]_{\alpha_1...\alpha_r}$ has a filtration:
 $$ R_0=\bC[x,y]\subset R_1\subset R_2=\bC[x,y]_{\alpha_1...\alpha_k}.$$
 (If $k=1$, then  $R_1=R_2$).
  By 
  localization and multiplication by $\tilde{\alpha}^{\tilde{\beta}}$ this induces a filtration by $A_2$-modules
 \begin{equation} \label{eq:dec} \bC[x,y]_{\tilde \alpha}\tilde{\alpha}^{\tilde{\beta}}\subset (R_1)_{\tilde \alpha}\tilde{\alpha}^{\tilde{\beta}}= (R_2)_{\tilde \alpha}\tilde{\alpha}^{\tilde{\beta}}=\bC[x,y]_{\alpha}\tilde{\alpha}^{\tilde{\beta}},
\end{equation}
 where the equality is a consequence of the fact that $R_2/R_1$ has support in the complement of $\tilde \alpha\neq 0$, and so vanishes when localised. By (\ref{eq: L}) the quotient 
 $$
 \bC[x,y]_{\alpha}\tilde{\alpha}^{\tilde{\beta}}/\bC[x,y]_{\tilde \alpha}\tilde{\alpha}^{\tilde{\beta}}=\bigoplus_{i=1}^k (L_{\alpha_i})_{\tilde \alpha}\tilde{\alpha}^{\tilde{\beta}},
  $$
as $A_2$-modules. We can describe the factors $K_i:=(L_{\alpha_i})_{\tilde \alpha}\tilde{\alpha}^{\tilde{\beta}}$ more precisely. Denote by $i_{H_i}$ the inclusion $H_i=V(\alpha_i)\to \bC^2$. Since $K_i$ has support on $H_i$, it is equal to ${i_{H_i}}_*(\rm{Ker_{\rm{K_i}}}\alpha_i)$, by Kashiwara's  theorem.
Note that $\alpha^c$ is a parameter on $H_i.$ The restriction of a form $\alpha_j, \ j=k+1,...,m$ to $H_i$ is a non-zero multiple of  $\alpha^c$, and hence intuitively the restriction of $\tilde \alpha^{\tilde\beta}$ is a multiple of $(\alpha^c)^{\beta_{k+1}+...+\beta_m}$.

\begin{lemma}\label{Thm55}Let $i_H:H=V(\alpha_i)\subset \bC^2$, and put $t=\alpha_i^c$, and let $ \beta_H=\beta_{k+1}+....+\beta_m$. Then \\
(i) $\rm{Ker_{K_i}\alpha_{i}}\cong\bC[t]_{t}t^{\beta_H}={M_{t}^{\beta_H}}$ as
 $\bC\langle t,\partial_{t}\rangle$-module.\\
 (ii) $\rm{K_{i}={i_{H}}_{*}M_{t}^{\beta_H}}$,
 where $\rm{i_{H}}$ is the inclusion $\bC\cong
 \rm{H}\subset\bC^{2}$. \end{lemma}
 \begin{proof}
It suffices to prove (i). We may without loss of generality assume that a basis is choosen so that $\alpha_i=x$ and $t=\alpha^c_i=y$, where the decomposition $\alpha_s=b_sx+c_sy, \ s=k+1,...,m$  has the property that $c_s\neq 0, \ s=k+1,...,m$(If some $c_s=0$, then $x$ and $\alpha_s$ are linearily dependent.). Since kernels localize well, $ \rm{Ker_{K_i}}= ( \rm{Ker_{L_i}}x)_{\tilde \alpha}\tilde{\alpha}^{\tilde{\beta}} $. Hence, by the description of $\rm{Im (L_{\alpha_i})}$ in (\ref{eq:L2}),
vector space generators of   $ \rm{Ker_{K_i}}x$ are given by
 $$
 \rm{Ker_{\rm{K_i}}}x=\bigoplus_{j\geq 0}\bC\frac{y^j}{x}\tilde{\alpha}^{\tilde{\beta}}+ \sum_{j\in\bZ, s=k+1,...,m}\bC\frac{1}{x\alpha_s^{j}}\tilde{\alpha}^{\tilde{\beta}}$$
 We may eliminate some generators, letting $\alpha_s=b_sx+c_sy$, and using that 
 $$
 \frac{c_s}{x\alpha_{s}}-\frac{c_r}{x\alpha_{r}}=0,
 $$ 
 (in $K_i$) and by an inductive argument getting $$
 \frac{c_s}{x\alpha^j_{s}}-\frac{c_r}{x\alpha^j_{r}}=0.
 $$
 This gives 
\begin{equation}
\label{eq:kerK}\rm{Ker_{\rm{K_i}}}x=\bigoplus_{j\geq 0}\bC\frac{y^j}{x}\tilde{\alpha}^{\tilde{\beta}}\oplus \bigoplus_{j\in\bZ}\bC\frac{1}{x\alpha_{k+1}^{j}}\tilde{\alpha}^{\tilde{\beta}}.
\end{equation}

 Let $\theta: \  \bC[x,y]\to \bC[y]$ be the surjection corresponding to the injection  $i_H:H=V(\alpha_i)\subset \bC^2$. Since $\alpha_s=b_sx+c_sy$ we have that $\theta(\alpha_s)=c_sy, \ c_s\in \bC$.  Choose a branch of the logarithm, so that $
 c=\prod_{s=k}^{m}c_{s}^{\beta_{s}}$ is defined.
 Define a $\bC [y]$- homomorphism  
 $$
 \theta:\rm{Ker_{\rm{K}}}x\longrightarrow\bC[y]_{y}y^{\beta_H}=M_y^{\beta_H},
 $$
 by $\theta(\frac{y^j}{x}\tilde{\alpha}^{\tilde{\beta}})=cy^jy^{\beta_{H}}$ and 
 $$\theta(\frac{1}{x\alpha_{s}^{j}}\tilde{\alpha}^{\tilde{\beta}})=c(c_sy)^{-j}y^{\beta_{H}}.$$
 
 The identity
 $$\frac{y}{x\alpha^j_s}=\frac{1}{c_s}\frac{b_sx+c_sy}{x\alpha^j_s}=\frac{1}{c_s}\frac{1}{x\alpha^{j-1}_s},
 $$
in $K_i$, implies that  $\theta(\frac{y}{x\alpha^j_s})=y\theta(\frac{1}{x\alpha^j_s})$. Furthermore
the identity 
$$\theta(\partial_y(\frac{1}{x}{\tilde{\alpha}^{\tilde{\beta}}}))=\theta(\frac{1}{x}\sum_{s=k+1}^m\frac{\beta_sc_s}{\alpha_s}{\tilde{\alpha}^{\tilde{\beta}}})=c(\sum_{s=k+1}^m\frac{\beta_sc_s}{c_sy}
 y^\beta_H)=\partial_y(\theta(\frac{1}{x}{\tilde{\alpha}^{\tilde{\beta}}})),$$
 shows part of the fact that $\theta$ commutes with the action of $\partial_y$ and the rest follows similarily.
 Clearly $\theta$ is onto, and injective by the description (\ref{eq:kerK}). This finishes the proof.
  \end{proof}
 
 This finally enables us to complete the proof of a more precise version of Theorem~\ref{MT}.
   .
 \begin{thm}Assume that $\beta_{1},..., \beta_k\in\bZ$ and
 $\beta_{k+1},...,\beta_{m}\in\bC\setminus\bZ$.\\
 (i) 
 If $\sum_{i=1}^{m}\beta_{i}\in\bC\setminus\bZ$, then
 $c(\rm{M_{\alpha}^{\beta}})=k+1$. There is one decomposition factor with support on the whole space, $1$  with support on each $H_s,\ s=1,...,k$, and none with support at the origin. \\
  (ii) 
  If $\sum_{i=1}^{m}\beta_{i}\in\bZ$, then $m-k\geq 2$ and
 $c(\rm{M_{\alpha}^{\beta})}=m+k-1$.There is one decomposition factor with support on the whole space, one with support on each $H_s,\ s=1,...,k$, and $m-k-2$ with support at the origin. 
 \end{thm}
 \begin{proof}By (\ref{eq:dec}) and Lemma \ref{Prop16}
 $$c(\rm{M}_{\alpha}^{\beta})=c(\bC[x,y]_{\tilde{\alpha}}{\tilde{\alpha}}^{\tilde{\beta}})+c(\rm{M}_{\alpha}^{\beta}/\bC[x,y]_{\tilde{\alpha}}{\tilde{\alpha}}^{\tilde{\beta}})=c(\bC[x,y]_{\tilde{\alpha}}{\tilde{\alpha}}^{\tilde{\beta}})+\sum_{s=1}^kc(\bC[y]_{y}y^{\beta_{H}}).$$
 
 Clearly $\beta_{H}\in\bZ\iff \vert \beta\vert =\sum_{i=1}^m\beta_i\in\bZ$, so $c(\bC[y]_{y}y^{\beta_{H}})=2$ or $1$, according to if $\sum_{i=1}^m\beta_i\in\bZ$, or not, by the example in the introduction. Hence the contribution of the last part of the sum is that there is always exactly one decomposition factor with support on each $H_i,\ i=1,...,k$. An additional $k$ with support at the origin exist if  $\vert \beta\vert \in\bZ$, while if $\vert \beta\vert \in\bC\setminus \bZ$ there are none. 
 
 On the other hand, by the previous section,  $c(\bC[x,y]_{\tilde{\alpha}}{\tilde{\alpha}}^{\tilde{\beta}})=(m-k-1)$ or $1$, according to if $\sum_{s=k}^m\beta_s\in\bZ\iff \sum_{i=1}^m\beta_i\in\bZ$, or not. There is always one decomposition factor with support on $\bC^2$, and $m-k-2$ with support in the origin if  $\sum_{i=1}^m\beta_i\in\bZ$.
 Combining these descriptions, if $\vert \beta\vert \in\bZ$ there are $2k+(m-k-1)=m+k-1$ composition factors, and if $\sum_{i=1}^m\beta_i\in\bC\setminus\bZ$, there are $k+1$.
\end{proof}

 \section*{Acknowledgements}
We would like to thank Jan-Erik Bj\"ork and Rolf K\"allstr\"om
for their interest in and crucial contributions to this note.


\begin{thebibliography}{99}
\bibitem{AM1}
 {Atiyah, M. F. and Macdonald, I. G.},
      {\em Introduction to commutative algebra},
 {Addison-Wesley Publishing Co., Reading, Mass.-London-Don
              Mills, Ont.},
       {1969}.
\bibitem{BJE1}
 {Bj{\"o}rk, J.-E.},
     {\em Rings of differential operators},
 {North-Holland Publishing Co.},
   {Amsterdam},
       {1979}.
       \bibitem{JEB2}
{Bj{\"o}rk, Jan-Erik},
     {\em Analytic $D$-modules and applications},
   {247},
 {Kluwer Academic Publishers Group},
    {Dordrecht},
       {1993}.
 \bibitem{CSC}
{Coutinho, S. C.},
     {\em A primer of algebraic {$D$}-modules},
        {33},
  {Cambridge University Press},
  {Cambridge},
      {1995}.
      \bibitem{CDCP1}
{C. De Concini and C. Procesi}, {\em The algebra of the Box-spline},
{arXiv:\textbf{math.NA/0602019} v1 (2006) }.

\bibitem{CDCP2}
 {C. De Concini and C. Procesi},
 {\em Topics in hyperplane arrangements, polytopes and box-splines},
 {book to be published}.
\bibitem{CLO}
{Cox, David and Little, John and O'Shea, Donal},
    {\em Ideals, varieties, and algorithms},
  {Springer},
 {New York},
      {2007}.
\bibitem{DJ1}
      {Dixmier, Jacques},
    {\em Enveloping algebras},
       {11},
      {American Mathematical Society},
  {Providence, RI},
     {1996},
\bibitem{FW1}
      {Fulton, William},
    {\em Algebraic curves},
 {Addison-Wesley Publishing Company Advanced Book Program},
  {Redwood City, CA},
      {1989}.
      


\bibitem{MH1}
{Matsumura, Hideyuki},
     {\em Commutative ring theory},
  {Cambridge University Press},
    {Cambridge},
     {1986}.

 \bibitem{SST}
{Saito, Mutsumi and Sturmfels, Bernd and Takayama, Nobuki},
    {\em Gr\"obner deformations of hypergeometric differential
              equations},
  {Springer-Verlag},
    {Berlin},
       {2000}.
       
      \bibitem{B}
        {Be{\u\i}linson, A. A.},
     {\em How to glue perverse sheaves},
 {{$K$}-theory, arithmetic and geometry ({M}oscow, 1984--1986)},
   {Lecture Notes in Math.} {1289},
    {(42--51)},
 {Springer},
{Berlin},
 {1987}.
        
        \bibitem{KV}
        {Khoroshkin, S. and Varchenko, A.},
      {\em Quiver {$D$}-modules and homology of local systems over
              an arrangement of hyperplanes},
{IMRP Int. Math. Res. Pap.},
   {2006}.
        
  \bibitem{Walter}
        {Walther, Uli},
      {\em Bernstein-{S}ato polynomial versus cohomology of the {M}ilnor
              fiber for generic hyperplane arrangements},
  {Compos. Math.},
  {141, No. 1},
 {2005},
{(121--145)}.
        
       \bibitem{Sabbah}
         {Sabbah, C.},
{\em Polyn\^omes de {B}ernstein-{S}ato \`a plusieurs variables},
  {S\'eminaire sur les \'equations aux d\'eriv\'ees partielles
              1986--1987},
 {Exp.\ No.\ XIX, 6},
 {1987},
	

\end{thebibliography}
\end{document}